\documentclass[10pt, reqno]{amsart}
\usepackage{graphicx, amssymb, amsmath, amsthm}
\numberwithin{equation}{section}

\usepackage{color}

\let\Im=\undefined\DeclareMathOperator*{\Im}{Im}

\newcommand{\N}{\mathcal{N}}

\newtheorem{theorem}{Theorem}[section]

\newtheorem{lemma}[theorem]{Lemma}

\theoremstyle{definition}

\newtheorem{remark}[theorem]{Remark}

\makeatletter
\newcommand{\Extend}[5]{\ext@arrow0099{\arrowfill@#1#2#3}{#4}{#5}}
\makeatother

\begin{document}
\title[Multilinear distorted multiplier estimate]{on   multilinear distorted  multiplier estimate and its applications}

\author{Kailong Yang}
\address{Chongqing National Center for Applied Mathematics, Chongqing Normal University, Chongqing, 401131, P.R.China}
\email{ykailong@mail.ustc.edu.cn}


\begin{abstract}
In this article, we investigate the  multilinear distorted  multiplier  estimate (Coifman-Meyer type theorem) associated with the Schr\"{o}dinger operator $H=-\Delta + V$ in the framework of the corresponding distorted Fourier transform. Our result is the ``distorted" analog of the multilinear Coifman-Meyer multiplier operator theorem in \cite{CM1}, which extends the bilinear estimates of Germain, Hani and Walsh's in \cite{PZS} to multilinear case  for all dimensions. As applications, we give the estimate of Leibniz's law of integer order derivations for the  multilinear distorted  multiplier for the first time and we obtain small data scattering for a kind of generalized mass-critical NLS with good potential in low dimensions $d=1,2$.

\end{abstract}

 \maketitle

\begin{center}
 \begin{minipage}{100mm}
   { \small {{\bf Key Words:}  multilinear estimate,  Schr\"{o}dinger operator, distorted Fourier transform, scattering, nonlinear Schr\"{o}dinger equation.}
      {}
   }\\
    { \small {\bf AMS Classification:}
      {35P25,  35Q55, 42B15, 42B10.}
      }
 \end{minipage}
 \end{center}


\section{Introduction}

\noindent

The study of multilinear pseudodifferential operators goes back to the pioneering works of $\mathrm{R}$. Coifman and Y. Meyer \cite{CM1,CM2,CM3,CM4}, since then, there has been a large amount of work on various generalisations of their results,  we will only make a rough list here.  After Lacey and Thiele's work \cite{MC1,MC2} on the bilinear Hilbert transform, with different assumptions on the symbols, the boundedness of multilinear operators in harmonic analysis in the classical Fourier transform setting have been well studied by many authors, for example, B\'{e}nyi and Torres \cite{AR1}, Gilbert and Nahmod \cite{JA1}, Grafakos and Kalton \cite{LN1}, Grafakos and Torres \cite{GT1}, Kenig and Stein \cite{CE1}  Muscalu, Tao and Thiele \cite{CT1} and  Tomita\cite{NT1}.

Some of the multilinear operators studied above are multilinear multipliers defined in the framework of the classical Fourier transform, and the classical Fourier transform of a function can be regarded as the projection into the eigenfunctions space of the  absolutely continuous spectrum of Laplacian operator $-\Delta$. For a given potential $V: \mathbb{R}^{d} \rightarrow \mathbb{R}$, consider the associated Schr\"{o}dinger operator $H:=$ $-\Delta+V$. When $V \in L^{2}\left(\mathbb{R}^{d}\right), H$ can be realized as a self-adjoint operator on $L^{2}\left(\mathbb{R}^{d}\right)$ with domain $D(H)=H^{2}\left(\mathbb{R}^{d}\right)$. We can impose a compactness condition on the multiplication operator associated with $V$ so that  the spectral properties of $H$ resemble those of $H_{0}=-\Delta$. We say that $V$ is short-range (or, of class SR) provided that
$$u \in H_{x}^{2}\left(\mathbb{R}^{d}\right) \mapsto(1+|x|)^{1+\epsilon} V u \in L_{x}^{2}\left(\mathbb{R}^{d}\right) \quad \text{is a compact operator},$$
for some $\epsilon>0$. It was shown by Agmon (cf.\cite{AS1}) that, for $V$ of class SR, $\sigma(H)=\left\{\lambda_{j}\right\}_{j \in J} \cup$ $[0, \infty)$; the continuous spectrum being $[0, \infty)$, and the discrete spectrum consisting of a countable set of real eigenvalues $\left\{\lambda_{j}\right\}$, each of finite multiplicity. Furthermore, we have the orthogonal decomposition
$$
L^{2}\left(\mathbb{R}^{d}\right)=L_{\mathrm{ac}}^{2}\left(\mathbb{R}^{d}\right) \oplus L_{\mathrm{p}}^{2}\left(\mathbb{R}^{d}\right),
$$
where $L_{\mathrm{p}}^{2}\left(\mathbb{R}^{d}\right)$ is the span of the eigenfunctions corresponding to the eigenvalues $\left\{\lambda_{j}\right\}$. and $L_{\mathrm{ac}}^{2}$ is the absolutely continuous subspace for $H$. Then we may try to define distorted Fourier transform on the absolutely continuous subspace for $H$, and multilinear distoerted multiplier, see Theorem \ref{ykl_th1} and \eqref{eqy1.11} respectively  for details below.

We investigate the estimate of the multilinear distorted multiplier(Coifman-Meyer type theorem) associated with the Schr\"{o}dinger operator $H=-\Delta + V$ in the framework of the corresponding distorted Fourier transform. As we know, there is only a small amount of research on this topic.  Germain, Hani and Walsh \cite{PZS} in 2015 investigated the bilinear estimates and applied it to the 3d quadratic nonlinear  Schr\"{o}dinger equations with a potential $V(x)$ to get global wellposedness for small data, in 2020, for the different kind of quadratic nonlinear terms, Pusateri and Soffer\cite{FA1} developed the corresponding bilinear estimates and used them to obtain similar results for the nonlinear  Schr\"{o}dinger equations with large potentials. Motivated by their work, we want to generalize the bilinear estimates to the multilinear case in this paper.

\subsection{Assumptions on the potential $V$}
Before stating our main results, let us now describe the assumptions we shall impose on $V$ as Germain, Hani and Walsh did in \cite{PZS}:

\begin{itemize}
  \item H1. Existence of distorted Fourier analysis (see remark \ref{yre1.1.2} for precise meaning below).
  \item H2. Absence of discrete spectrum for $-\Delta+V$.
  \item H3. $L^{p}$ boundedness of the wave operator $\Omega:=\lim\limits _{t \rightarrow-\infty} \mathrm{e}^{\mathrm{i} t H} \mathrm{e}^{-\mathrm{i} t \Delta}$.
\end{itemize}

In total, H1, H2, H3 amount to some regularity and decay requirements on V.
At some points in our analysis, we  need the boundedness of $\Omega$ in high-order regular Sobolev spaces.  In those cases, we assume:

\begin{itemize}
  \item $\mathrm{H3}^*$. $W^{1,p}$ boundedness of the wave operator $\Omega $.
\end{itemize}
For a more specific discussion of hypotheses  H1 and  H2, see remark \ref{yre1.1.1} below. For the discussion of hypotheses H3 and $\mathrm{H3}^*$, see Theorem \ref{yth1.1.1}, \ref{yth1.1.2} and \ref{yth1.1.3}.

\subsection{Distorted Fourier transform}

 For each $\xi \in \mathbb{R}^{d} \backslash\{0\},$ we know that for $V$ of class SR, $|\xi|^{2}$ is in the continuous spectrum of $H$, the associated eigenfunction is the distorted plane wave $e(\cdot ; \xi)$ defined as the solution of
\begin{equation}\label{ykl3_1}
H e(\cdot ; \xi)=|\xi|^{2} e(\cdot ; \xi)
\end{equation}
with the asymptotic condition
$$
v(x ; \xi):=e(x ; \xi)-\mathrm{e}^{\mathrm{i} x \cdot \xi}=O\left(|x|^{-1}\right) \quad \text { as }|x| \rightarrow \infty
$$
and the Sommerfeld radiation condition
$$
\lim _{r \rightarrow \infty} r\left(\partial_{r}-\mathrm{i}|\xi|\right) v=0.
$$
The eigenfunction problem \eqref{ykl3_1} can be recast as the Lippman-Schwinger equation:
\begin{equation}\label{ykl3_2}
e(\cdot ; \xi)=e_{\xi}-R_{V}^{-}\left(|\xi|^{2}\right) V e_{\xi}, \quad e_{\xi}(x):=\mathrm{e}^{\mathrm{i} x \cdot \xi},
\end{equation}
where $R_{V}^{-}(z):=\lim _{\epsilon \rightarrow 0+} R_{V}(z-\mathrm{i} \epsilon)=\lim _{\epsilon \rightarrow 0+}(H-(z-i\epsilon))^{-1}$.
It can be shown that there exists a unique solution to \eqref{ykl3_2} for any $\xi \in \mathbb{R}^{d} \backslash\{0\}$ provided that $V=O\left(|x|^{-1-\epsilon}\right)$ as $|x| \rightarrow \infty$, for some $\epsilon>0$ (cf. \cite{AS1}). Under this assumption, the distorted plane waves are relatively smooth in $x$, but have very little regularity in $\xi$. More precisely, for fixed $\xi \in \mathbb{R}^{d} \backslash\{0\}$,
\begin{equation}\label{ykl3_3}
e(\cdot ; \xi) \in\langle x\rangle^{s} H_{x}^{2}, \quad \text { for any } s>(d+1) / 2,
\end{equation}
however, the map $(x, \xi) \mapsto e(x ; \xi)$ is merely measurable. One can improve this by requiring additional decay and regularity of $V$ (cf., e.g., \cite{IT1}).

In view of the Fourier transform, we expect that the family $\{e(\cdot ; \xi)\}$ forms a basis for the absolutely continuous subspace of $H$. This is indeed true, as was first proved by Ikebe \cite{IT1} and later generalized by several authors. For consistency of presentation, we give here the version due to Agmon (cf. \cite{AS1}, Theorem 6.2). Before that, let us now impose assumption H2, namely that $H$ has no discrete spectrum. However, we remark that many results in this paper can be directly generalized to potentials with discrete eigenvalues by simply projecting on the absolutely continuous subspace $L_{\mathrm{ac}}^{2}$ throughout. That said, the result is the following.

\begin{theorem}[Tkebe, Alsholm-Schmidt, Agmon\cite{AS1,IT1}]\label{ykl_th1}
 Consider the Schr\"{o}dinger operator $H$ with potential $V$ satisfying $\mathrm{H} 2$ and
\begin{equation}\label{ykl3_4}
(1+|x|)^{2(1+\epsilon)} \int_{B_{1}(x)}|V(y)|^{2}|y-x|^{-d+\theta} \mathrm{d} y \in L_{x}^{\infty}\left(\mathbb{R}^{d}\right) \quad \text { for some } \epsilon>0,0<\theta<4.
\end{equation}
Define the distorted Fourier transform $\mathcal{F}^{\sharp}$ by
\begin{equation}\label{ykl3_5}
\left(\mathcal{F}^{\sharp} f\right)(\xi):=f^{\sharp} (\xi):=\frac{1}{(2 \pi)^{d / 2}} \lim _{R \rightarrow \infty} \int_{B_{R}} \overline{e(x ; \xi)} f(x) d x,
\end{equation}
where $B_{R}$ is the ball of radius $R$ centered at the origin in $\mathbb{R}^{d}$. Then, $\mathcal{F}^{\sharp}$ is an isometric isomorphism on $L^{2}\left(\mathbb{R}^{d}\right)$ with inverse formula
\begin{equation}\label{ykl3_6}
f(x)=\left(\mathcal{F}^{\sharp^{-1}} f^{\sharp}\right)(x):=\frac{1}{(2 \pi)^{d / 2}} \lim _{R \rightarrow \infty} \int_{B_{R}} e(x ; \xi) f^{\sharp}(\xi) \mathrm{d} \xi,
\end{equation}
Moreover, $\mathcal{F}^{\sharp}$ diagonalizes $H$ in the sense that, for all $f \in H^{2}\left(\mathbb{R}^{d}\right)$,
\begin{equation}\label{ykl3_7}
H f=\mathcal{F}^{\sharp^{-1}} M \mathcal{F}^{\sharp} f,
\end{equation}
where $M$ is the multiplication operator $u \mapsto|x|^{2} u$.
\end{theorem}
\begin{remark}\label{yre1.1.2}
We are now able to give a precise meaning to assumption $\mathrm{H} 1$: $\mathrm{H} 1$ is satisfied provided that
\begin{enumerate}
  \item The family of eigenfunctions $\{e(\cdot, \xi)\}$ exists with the regularity stated in \eqref{ykl3_3};
  \item The operator $\mathcal{F}^{\sharp}$ defined by \eqref{ykl3_5} exists and exhibits the properties described in Theorem \ref{ykl_th1}.
\end{enumerate}
\end{remark}

Once we have defined the distorted Fourier transform, for any function $m: \mathbb{R}^d\rightarrow \mathbb{C}$, we define the distorted Fourier multiplier $m(D^{\sharp})$ to be the operator,
\begin{equation}\label{eq1.11}
 m(D^{\sharp}):= \mathcal{F}^{\sharp^{-1}}m(\xi)\mathcal{F}^{\sharp},
\end{equation}
this is an analog of the well-studied Fourier multiplier $m(\nabla)$ given by $m(\nabla):= \mathcal{F}^{-1}m(\xi)\mathcal{F}$. For us, the importance of $\Omega$ lies in the intertwining relations
\begin{equation}\label{eq1.12}
 \mathrm{e}^{\mathrm{i} t H}=\Omega \mathrm{e}^{\mathrm{i} t H_{0}} \Omega^{*}, \quad \mathcal{F}^{\sharp} \Omega=\mathcal{F}, \quad m(D^{\sharp})= \Omega m(\nabla) \Omega^{*},
\end{equation}
In other words, $\Omega$ allows us to translate back and forth between the flat and distorted
cases. Clearly, then, information about the structure and boundedness properties of $\Omega$
is extremely valuable. We collect some results about the boundedness properties of $\Omega$  below for different dimensions. 

\begin{theorem}[\cite{FK1,Yajima1,Yajima2,Yajima3,JY1,Weder}] \label{yth1.1.1}\mbox{}
\begin{enumerate}
  \item  $d \geq 3 .$ Yajima, Finco-Yajima \cite{FK1,Yajima1,Yajima2,Yajima3}. Let $k \in \mathbb{N}$ and consider the Schr\"{o}dinger operator $H$ with real potential $V: \mathbb{R}^{d} \rightarrow \mathbb{R}$ for $d \geq 3 .$ Fix $p_{0}, k_{0}$ as follows:
$$
\left\{\begin{array}{ll}
p_{0}=2, k_{0}=0 & \text { if } d=3 \\
p_{0}>d / 2, k_{0}:=\lfloor(d-1) / 2\rfloor & \text { if } d \geq 4
\end{array}\right.
$$
Assume that for some $\delta>(3 d / 2)+1$,
\begin{equation}\label{ykl3_eq8}
\langle x\rangle^{\delta}\left\|\partial^{\alpha} V\right\|_{L_{y}^{p_{0}}(|x-y| \leq 1)} \in L_{x}^{\infty}\left(\mathbb{R}^{d}\right), \quad \text{for all $\alpha$ with $|\alpha| \leq k+k_{0}$}.
\end{equation}
Then, $V$ is of class SR and so $\Omega$ and $\Omega^{*}$ are well-defined as operators on $L^{2}\left(\mathbb{R}^{d}\right) \cap$ $W^{k, p}\left(\mathbb{R}^{d}\right)$. If we additionally assume that $V$ is of Generic-type:
\begin{equation}\label{ykl3_eq9}
\text{ there is no $u \in\langle x\rangle^{\theta} L_{x}^{2}\left(\mathbb{R}^{d}\right)$ solving $H u=0, \quad$ for any $\theta>\frac{1}{2}$.}
\end{equation}
Then, $\Omega$ and $\Omega^{*}$ may be extended to bounded operators defined on $W^{k, p}\left(\mathbb{R}^{d}\right)$.
  \item $d=2$. Jensen, Yajima \cite{JY1}. Suppose that $V(x)$ is real-valued and $|V(x)| \leq C\langle x\rangle^{-\delta}, x \in \mathbf{R}^{2}$, for some $\delta>6$, and 0 is neither an eigenvalue nor a resonance of $H$, viz. there are no solutions $u \in H_{\mathrm{loc}}^{2}\left(\mathbf{R}^{2}\right) \backslash\{0\}$ of $-\Delta u+V u=0$, which for some $\alpha$, $b_{1}$, and $b_{2}$ satisfy for $|\alpha| \leq 1$,
$$
\partial_{x}^{\alpha}\left(u-a-\frac{b_{1} x_{1}+b_{2} x_{2}}{|x|^{2}}\right)=O\left(|x|^{-1-\varepsilon-|\alpha|}\right), \quad|x| \rightarrow \infty
$$
Then the wave operators $\Omega$ are bounded in $L^{p}\left(\mathbf{R}^{2}\right)$ for all $p, 1<p<\infty$. moreover, the boundedness of wave operator $\Omega$ in the sobolev space $W^{k,p}(\mathbb{R}^2)$ can be obtained by applying the commutator method for any $1<p<\infty$, $k=0,\cdots,l$,  if $V$ satisfies $|D^{\alpha}V(x)|\leq C_{\alpha} \langle x\rangle^{-\delta} $ for $|\alpha|\leq l$ and 0 is neither an eigenvalue nor a resonance of $H$.
  \item $d=1$, Weder \cite{Weder}. Let $f_{j}(x, k), j=1,2, \Im k \geq 0$, be the Jost solutions to the following equation
\begin{equation}\label{y1.1111}
-\frac{d^{2}}{d x^{2}} u+V u=k^{2} u, k \in \mathbf{C}
\end{equation}
let $[u, v]$ denote the Wronskian of $u$ and $v$ :
$
[u, v]:=\left(\frac{d}{d x} u\right) v-u \frac{d}{d x} v.
$
 A potential $V$ is said to be generic if $\left[f_{1}(x, 0), f_{2}(x, 0)\right] \neq 0$ and $V$ is said to be exceptional if $\left[f_{1}(x, 0), f_{2}(x, 0)\right]=0$. If $V$ is exceptional there is a bounded solution (a half-bound state, or a zero energy resonance) to \eqref{y1.1111} with $k=0$. For $l=0,1, \cdots$, we denote $V^{(l)}:=\frac{d^{l}}{d x^{l}} V(x)$. Note that $V^{(0)}=V$.  Suppose that $V \in L_{\gamma}^{1}$ with $\|V\|_{L_{\gamma}^{1}}:=\int|V(x)|(1+|x|)^{\gamma} d x$, where in the generic case $\gamma>3 / 2$ and in the exceptional case $\gamma>5 / 2$, and that for some $k=1,2, \cdots, V^{(l)} \in L^{1}$, for $l=$ $0,1,2, \cdots, k-1 .$ Then $\Omega$ and $\Omega^{*}$ originally defined on $W^{k, p} \cap L^{2}, 1 \leq p \leq \infty$ have extensions to bounded operators on $W^{k, p}, 1<p<\infty .$ Moreover, there are constants $C_{p}, 1<p<\infty$, such that:
\begin{equation}\label{y1.62}
\left\|\Omega f\right\|_{k, p} \leq C_{p}\|f\|_{k, p} ,\quad \left\|\Omega^{*} f\right\|_{k, p} \leq C_{p}\|f\|_{k, p}, 
\end{equation} $f \in W^{k, p} \cap L^{2}, 1<p<\infty$.
Furthermore, if $V$ is exceptional and $a:=\lim _{x \rightarrow-\infty} f_{1}(x, 0)=1, \Omega$ and $\Omega^{*}$ have extensions to bounded operators on $W^{k, 1}$ and to bounded operators on $W^{k, \infty}$, and there are constants $C_{1}$ and $C_{\infty}$ such that \eqref{y1.62} holds for $p=1$ and $p=\infty$.
\end{enumerate}

\end{theorem}

\subsection{The main results}
We start by considering the following  multilinear distorted  multiplier of the form:
\begin{equation}\label{eqy1.11}
\begin{split}
  T(f_1,f_2,\ldots,f_k)(x):=&\int_{\mathbb{R}^{d}} \ldots\int_{\mathbb{R}^{d}} m\left(\xi_{1}, \xi_{2},\ldots,\xi_k\right) f_1^{\sharp}\left(\xi_{1}\right) f_2^{\sharp}\left(\xi_{2}\right)\ldots f^{\sharp}_k\left(\xi_{k}\right)\\
 &\qquad e\left(x, \xi_{1}\right) e\left(x, \xi_{2}\right)\ldots e\left(x, \xi_{k}\right) \mathrm{d} \xi_{1} \mathrm{d} \xi_{2} \ldots \mathrm{d} \xi_{k}.
\end{split}
\end{equation}

When $e(x,\xi_j)=e^{ix\xi_j}$, $f_j^{\sharp}=\hat{f}_j$, The multilinear distorted multiplier defined above becomes the Coifman-Meyer multilinear multiplier. Note that the case $m=1$ corresponds (up to a constant factor) to the product of $f_1,\ldots,f_k$. We say that the multiplier $m$ satisfies Coifman-Meyer type bounds if the following
homogeneous bounds hold for sufficiently many multi-indices $\alpha_1$, $\alpha_2, \ldots,\alpha_k$:
\begin{equation}\label{1.3.1}
\left|\partial_{\xi_{1}}^{\alpha_1} \partial_{\xi_{2}}^{\alpha_2} \ldots \partial_{\xi_{k}}^{\alpha_k} m\left(\xi_{1}, \xi_{2},\ldots,\xi_k\right)\right| \leq C\left(\left|\xi_{1}\right|+\left|\xi_{2}\right| +\ldots+\left|\xi_{k}\right|\right)^{-(|\alpha_1|+|\alpha_2|+\ldots+|\alpha_k|)}
\end{equation}

Our result is the distorted analog of the multilinear Coifman-Meyer multiplier operator theorem in \cite{CM1}, but with a little integrability index destruction when we have no assumption on the $L^p$ boundedness of the Riesz transform \(\Re=\nabla(-\Delta+V)^{-1 / 2}\).
\begin{theorem}\label{mthm1.1'}
	For $d \geq 1$, let $V \in L^{p_V}\left(\mathbb{R}^{d}\right)$ be a potential satisfying \(\mathrm{H} 1, \mathrm{H} 2,\) and \(\mathrm{H} 3 \) for $p_V:=\frac{d}{s_1+s_2}$ with $s_1,s_2$ defined in \eqref{req1.3.5'} below.  Suppose that \(m\left(\xi_{1}, \xi_{2},\ldots,\xi_k\right)\) is a Coifman-Meyer symbol in $k$ variables as in \eqref{1.3.1}, then
\begin{enumerate}
  \item for \( p_j, r^{\prime}\in(1, \infty), j=1,\ldots,k\) satisfy $\frac{1}{r^{\prime}}=\sum^k_{j=1}\frac{1}{p_j}$,
\begin{equation}\label{ykl1.1'}
 \|T(f_1,\ldots,f_k)\|_{L^{r^\prime}\left(\mathbb{R}^{d}\right)} \lesssim_{m, V}\Pi_{j=1}^{k}\|f_j\|_{L^{p_j}\left(\mathbb{R}^{d}\right)}.
\end{equation}
provided that the Riesz transform \(\Re=\nabla(-\Delta+V)^{-1 / 2}\) is bounded on $L^{p_j}, j=1,\ldots,k$ and $L^{r^\prime}$.
  \item suppose instead that $V$ satisfies \(\mathrm{H} 3^{*} \),  for \( p_j, r^{\prime},  \tilde{p}_j\in(1, \infty), j=1,\ldots,k\) satisfying \(\frac{1}{r^{\prime}}=\sum^k_{j=1}\frac{1}{p_j}=\sum^k_{j=1}\frac{1}{\tilde{p}_j} -\frac{\epsilon}{d}\) for some \(\epsilon>0\),
\begin{equation}\label{ykl1.1}
 \|T(f_1,\ldots,f_k)\|_{L^{r^\prime}\left(\mathbb{R}^{d}\right)} \lesssim_{m, V}\Pi_{j=1}^{k}\|f_j\|_{L^{p_j}\left(\mathbb{R}^{d}\right)}+
      \Pi_{j=1}^{k}\|f_j\|_{L^{\tilde{p}_j}\left(\mathbb{R}^{d}\right)}.
\end{equation}
\end{enumerate}
\end{theorem}
\begin{remark}[More discussions of the conditions for theorem \ref{mthm1.1'}]\label{yre1.1.1}\mbox{}
\begin{enumerate}
  \item {\bf{About assumption H1:}} It follows that sufficient conditions for $\mathrm{H} 1$ are that $V$ satisfies $\mathrm{H} 2$, \eqref{ykl3_4}
and $V=O\left(|x|^{-1-\epsilon}\right)$ as $|x| \rightarrow \infty$, for some $\epsilon>0$.
 \item {\bf{About assumption H2:}} Once \eqref{ykl3_4} is satisfied, we rule out the existence of nonnegative eigenvalues. Additionally if the negative part of $V$ is not very large, there are no negative eigenvalues(e.g., if $d \geq 3$, Hardy's inequality implies that the condition $V \geq-(d-2)^{2} / 4|x|^{2}$ is sufficient to rule out both nonpositive eigenvalues and resonances at 0 as defined in \eqref{ykl3_eq9}), then $\mathrm{H} 2$ holds.
  \item {\bf{About Riesz transform:}} (a). If $V$ belongs to the class of $B_{d}$, by theorem 1.2 in \cite{AB1},  the Riesz transform \(\Re=\nabla(-\Delta+V)^{-1 / 2}\) is bounded on $L^{p}, 1<p<\infty$. Here, in \cite{AB1}, $B_{q}, 1<q \leq \infty$, is the class of the reverse H\"{o}lder weights: $w \in B_{q}$ if $w \in L_{l o c}^{q}\left(\mathbb{R}^{d}\right), w>0$ almost everywhere and there exists a constant $C$ such that for all cube $Q$ of $\mathbb{R}^{d}$,
$$
\left(\frac{1}{|Q|} \int_{Q} w^{q}(x) d x\right)^{1 / q} \leq \frac{C}{|Q|} \int_{Q} w(x) d x
$$
If $q=\infty$, then the left hand side is the essential supremum on $Q$.  Examples of $B_{q}$ weights are the power weights $|x|^{-\alpha}$ for $-\infty<\alpha<d / q$ and positive polynomials for $q=\infty$. \\ (b). Let $V$ be a real potential such that
$$
V \in C^{\infty}(\mathbb{R}^{d}), \quad V=O\left(x^{3}\right) \text { as } x \rightarrow 0.
$$ Moreover, assume $\left|V_{-}(z)\right| \leq \alpha|z|^{-2}$ for $\alpha<(d / 2-1)^{2}$ or $V_{-}=0$, where $V_{-}$ is the negative part of $V$, then $H=-\Delta+V$ does not have a zero-resonance nor nontrivial $L^{2}$ kernel, by theorem 1.3 in \cite{CGAH}, for $d\geq 3$, the Riesz transform \(\Re=\nabla(-\Delta+V)^{-1 / 2}\) is bounded on $L^{p}, 1<p<d$.
  \end{enumerate}
\end{remark}

\begin{remark}
When the symbol depends on  the spatial variable $x$, i.e. $m(\xi_1,\xi_2,\cdots,\xi_k)$ becomes $m(x;\xi_1,\xi_2,\cdots,\xi_k)$, we think that following the argument to the proof of theorem \ref{mthm1.1'}, we can get the same result if we add appropriate condition to $m(x;\xi_1,\xi_2,\cdots,\xi_k)$ on the spatial variable $x$ additionally. In particular, when $m(x;\xi_1,\xi_2,\cdots,\xi_k)=a(x)m(\xi_1,\xi_2,\cdots,\xi_k)$, we can get the same conclusion as theorem \ref{mthm1.1'} by just letting $a\in L^{\infty}$ additionally.
\end{remark}
Our results extend Germain, Hani and Walsh's bilinear estimates in \cite{PZS} to the multilinear case and hold for all dimensions $d\geq 1$. We think the assumptions  for the potential in our theorem can be weakened properly, and now the assumptions for the potential are not optimal. Note that our multilinear estimate can not be obtained directly from bilinear estimate by induction, because even in the framework of classical Fourier transform,  it can not be obtained.  In our proof, after the distorted-frequency localization, we do not divide the distorted-frequency region of multiple summations into upper and lower triangular regions roughly according to the symmetry, such as $\Lambda(f_1,\cdots,f_{k+1})=C \sum_{N_1<\cdots<N_{k+1}}\Lambda(f_{1;N_1},\cdots,f_{k+1;N_{k+1}}).$
where $k+1$-linear form $\Lambda(f_1,\cdots,f_k,f_{k+1})$ is defined in \eqref{eqy3.1.1}. While we divide the distorted-frequency region of multiple summations into high distorted-frequency and low distorted-frequency parts,  the low distorted-frequency part is  $$\Lambda_L(f_1,\cdots,f_{k+1}):=\sum\limits_{N_1\leq 1,\cdots,N_{k+1}\leq 1}\Lambda(f_{1;N_1},\cdots,f_{k+1;N_{k+1}})=\Lambda(f_{1;\leq1},\cdots,f_{k+1;\leq1}).$$  For the low distorted-frequency part, we eliminate the multiple summations and estimate directly, and then obtain the multilinear estimate without the destruction of integrability index.  In this way, we will not worry too much about the limitation of index and dimensions of Sobolev's embedding theorem in the subsequent proof. This removes the limitation of dimension; For the high distorted-frequency part, we partly decompose the multiple summations region into the upper and lower triangular regions, and make use of more symmetries and cancellations, so as to cut down the multiplicities of the summations about distorted-frequency and prevent the logarithmic divergence caused by multiple summations. Finally we obtain as desired, see section 3 for details.

Compared with the flat case of multilinear multipliers,  the difficulties we face come from the nonlinear spectral distribution when we try to estimate the  multilinear distorted  multipliers, here the nonlinear spectral distribution is given as follows, $$M(\xi_1,\cdots,\xi_k,\xi_{k+1})=\int_{\mathbb{R}^{d}} e(x,\xi_1)e(x,\xi_2)\cdots e(x,\xi_{k+1})dx.$$
in flat case, $M(\xi_1,\cdots,\xi_k,\xi_{k+1})=\delta(\xi_1+\cdots+\xi_k+\xi_{k+1})$. However, in the distorted case,  $M(\xi_1,\cdots,\xi_k,\xi_{k+1})\neq\delta(\xi_1+\cdots+\xi_k+\xi_{k+1})$,  we don't have convolution structure $\mathcal{F}^{\sharp}(fg)=\int f^{\sharp}(\xi-\eta)g^{\sharp}(\eta)d\eta$ any more. Therefore, we know little about the distorted-frequency support distribution of the multilinear distorted  multipliers, which makes it impossible for us to estimate the fractional derivatives of the  multilinear distorted  multipliers by Bony decomposition. We only obtain the estimates of integer derivatives of the  multilinear distorted  multipliers providing that the Riesz transform \(\Re=\nabla(-\Delta+V)^{-1 / 2}\) is bounded on $L^{p}, 1<p<\infty$, and  we obtain the estimates of even integer derivatives of the  multilinear distorted  multipliers if we do not have the assumption that the Riesz transform \(\Re=\nabla(-\Delta+V)^{-1 / 2}\) is bounded on $L^{p}, 1<p<\infty$.

\begin{theorem}[Leibniz's law of integer order derivations]\label{reyth1.1}
For $s\geq0$ an integer, $V\in W^{s,\infty}(\mathbb{R}^{d})\cap W^{s,d}(\mathbb{R}^{d}), d\geq1$, let $V$ satisfy \(\mathrm{H} 1, \mathrm{H} 2,\) and \(\mathrm{H} 3^{*} \),
\begin{enumerate}
  \item If the Riesz transform \(\Re=\nabla(-\Delta+V)^{-1 / 2}\) is bounded on $L^{p}, 1<p<\infty$. Then for  $\sum_{j=1}^k\frac{1}{p_j}=\frac{1}{r'}=1-\frac{1}{r}$,  we have
\begin{equation}\label{y3.2}
	\begin{split}
\|T(f_1,\cdots,f_k)\|_{\dot{W}^{s,r'}_{\sharp}}\lesssim \sum_{l=1}^{k}\|f_l\|_{\dot{W}^{s,p_l}_{\sharp}}\prod_{j=1,j\neq l}^k\|f_j\|_{L^{p_j}}+ \prod_{j=1}^k\|f_j\|_{L^{p_j}}.
	\end{split}
\end{equation}
and
\begin{equation}\label{y3.3}
	\begin{split}
\|T(f_1,\cdots,f_k)\|_{W^{s,r'}_{\sharp}}\lesssim \sum_{l=1}^{k}\|f_l\|_{W^{s,p_l}_{\sharp}}\prod_{j=1,j\neq l}^k\|f_j\|_{L^{p_j}}.
	\end{split}
\end{equation}
  \item If we do not have the assumption that the Riesz transform \(\Re=\nabla(-\Delta+V)^{-1 / 2}\) is bounded on $L^{p}, 1<p<\infty$.  For $s=2k\geq0$ being an even integer, \( p_j, r^{\prime},  \tilde{p}_j\in(1, \infty), j=1,\ldots,k\) satisfying \(\frac{1}{r^{\prime}}=\sum^k_{j=1}\frac{1}{p_j}=\sum^k_{j=1}\frac{1}{\tilde{p}_j} -\frac{\epsilon}{d}\) for some \(\epsilon>0\), we have
\begin{equation}\label{reyeq1.1}
	\begin{split}
\|T(f_1,\cdots,f_k)\|_{\dot{W}^{s,r'}_{\sharp}}&\lesssim \sum_{l=1}^{k}\|f_l\|_{\dot{W}^{s,p_l}_{\sharp}}\prod_{j=1,j\neq l}^k\|f_j\|_{L^{p_j}}+ \prod_{j=1}^k\|f_j\|_{L^{p_j}}\\
&+\sum_{l=1}^{k}\|f_l\|_{\dot{W}^{s,\tilde{p}_l}_{\sharp}}\prod_{j=1,j\neq l}^k\|f_j\|_{L^{\tilde{p}_j}}+ \prod_{j=1}^k\|f_j\|_{L^{\tilde{p}_j}}.
	\end{split}
\end{equation}
and
\begin{equation}
	\begin{split}
\|T(f_1,\cdots,f_k)\|_{W^{s,r'}_{\sharp}}&\lesssim \sum_{l=1}^{k}\|f_l\|_{W^{s,p_l}_{\sharp}}\prod_{j=1,j\neq l}^k\|f_j\|_{L^{p_j}}\\
&+\sum_{l=1}^{k}\|f_l\|_{W^{s,\tilde{p}_l}_{\sharp}}\prod_{j=1,j\neq l}^k\|f_j\|_{L^{\tilde{p}_j}}.
	\end{split}
\end{equation}
\end{enumerate}
\end{theorem}

\begin{remark}
\begin{enumerate}
  \item The second term on the right hand side of \eqref{y3.2} and \eqref{reyeq1.1} comes from the contribution of the term containing $V(x)$ in \eqref{y3.1} and \eqref{y3.1'}.
  \item Without the assumption that the Riesz transform \(\Re=\nabla(-\Delta+V)^{-1 / 2}\) is bounded on $L^{p}, 1<p<\infty$, we can only obtain the estimates for  the even integer  derivatives of the multilinear distorted multiplier, because we cannot estimate the $L^{r}$ boundedness of the Riesz transform of the test function $\nabla(-\Delta+V)^{-1/2}f_{k+1}$ in the duality formula \eqref{y3.1'} below.
\end{enumerate}
\end{remark}


As another application, we consider the following  generalized mass-critical nonlinear Schr\"{o}dinger equation with good potential in low dimensions $d=1,2$:
\begin{equation}\label{y3.2.1''}
i u_{t}-\Delta u+Vu=a(x)F(u), \quad u(0, x)=u_{0}(x), \quad x\in\mathbb{R}^d.
\end{equation}
when $d=1$, $F(u)=T(\bar{u},\bar{u},u,u,u)(x)$; when $d=2$, $F(u)=T(\bar{u},u,u)(x)$.
Note that the case symbol $m=1$ corresponds (up to a constant factor) to the product of the functions. Therefore, in this case, when $V=0$ and $a(x)\equiv1$, or $a(x)\equiv-1$, the equation \eqref{y3.2.1''} becomes a classical mass-critical nonlinear Schr\"{o}dinger equation in $d=1,2$.

 For good potential $V$: $V$ satisfies \(\mathrm{H} 1, \mathrm{H} 2,\) and \(\mathrm{H} 3 \), and assume that the Riesz transform \(\Re=\nabla(-\Delta+V)^{-1 / 2}\) is bounded on $L^{p}, 1<p<\infty$ when $d=1,2$, we have the scattering of the generalized mass-critical NLS with good potential for small data in low dimensions $d=1,2$.

\begin{theorem}[local wellposedness and  small data scattering]\label{th-y3.3} For $d=1,2$, $a(x)\in L^{\infty}$, the equation \eqref{y3.2.1''} has the following properties:
\begin{enumerate}
  \item (Local wellposedness) For any $u_{0} \in L_{x}^{2}\left(\mathbf{R}^{d}\right)$, there exists $T\left(u_{0}\right)>0$ such that \eqref{y3.2.1''} is locally well posed on $[-T, T].$ The term $T\left(u_{0}\right)$ depends on the profile of the initial data as well as its size. Moreover, \eqref{y3.2.1''} is well posed on an open interval $I \subset \mathbf{R}$, $0 \in I$;
  \item (Small data scattering) there exists $\varepsilon_{0}(d)>0$, such that if
\begin{equation}
\left\|u_{0}\right\|_{L^{2}\left(\mathbf{R}^{d}\right)} \leq \varepsilon_{0}(d),
\end{equation}
then \eqref{y3.2.1''} is globally well posed and scattering, i.e. there exist $u^{\pm} \in L_{x}^{2}(\mathbb{R}^d )$ such that
\begin{equation}
\|u(t)- e^{it\Delta} u^{\pm}\|_{L_{x}^{2}} \to 0, \ \text{ as } t\to \pm \infty.
\end{equation}
\end{enumerate}
\end{theorem}
Finally, we organize the paper as follows: we list some notations and basic lemmas  in section 2.  In the third section, we give the proof of the main results, and the fourth section are the applications, including the estimate of Leibniz's law of integer order derivations for the  multilinear distorted  multiplier and small data scattering for a kind of generalized mass-critical NLS with good potential in low dimensions $d=1,2$.



\section{preliminary}
We will use the notation $X\lesssim Y$ whenever there exists some constant $C>0$ so that $X \le C Y$. Similarly, we will use $X \sim Y$ if
$X\lesssim Y \lesssim X$. Also, we use the Japanese bracket convention $\langle x\rangle^2:=1+|x|^2$.
Let $\psi \in C^{\infty}_0(\mathbb{R}^d)$ be a radial, decreasing function
\begin{equation}\label{ykl1.11}\psi(x):=
\left\{
  \begin{array}{ll}
    1, & \hbox{$|x|\leq 1$,} \\
    0, & \hbox{$|x|>2$.}
  \end{array}
\right.
\end{equation}
For $N\in2^{\mathbb{Z}}$, we denote
\begin{equation*}
\psi(\frac{x}{N})-\psi(\frac{2x}{N})=:\phi(\frac{x}{N}).
\end{equation*} The Littlewood-Paley operators are then given by
$$P^{\sharp}_{N}=\phi\left(\frac{D^{\sharp}}{N}\right) \quad \text { and } \quad P^{\sharp}_{<N}=\psi\left(\frac{D^{\sharp}}{N}\right) \quad N \in 2^{\mathbb{Z}},$$
then we have distorted-frequency decomposition,  $f=\sum\limits_{N\in 2^{\mathbb{Z}}}P^{\sharp}_{N}f.$

We define distorted sobolev norm as $\|f\|_{\dot{W}^{k,p}_{\sharp}}:=\||D^{\sharp}|^{k}f\|_{L^p}=\|H^{k/2}f\|_{L^p}$ and $\|f\|_{W^{k,p}_{\sharp}}:=\|f\|_{L^p}+\|f\|_{\dot{W}^{k,p}_{\sharp}}.$ then
\begin{lemma}[\cite{PZS}]For $1<p<\infty$ and $V$ satisfying \(\mathrm{H} 1, \mathrm{H} 2,\) and \(\mathrm{H} 3^{*} \), we have
$$
\|f\|_{\dot{W}^{k,p}_{\sharp}} \sim_{d}\left\|\left(\sum_{N \in 2^{\mathbb{Z}}} N^{2 s}\left|P^{\sharp}_{N} f\right|^{2}\right)^{1 / 2}\right\|_{L_{x}^{p}}, \quad \|f\|_{W^{k,p}_{\sharp}} \sim_{d}\left\|\left(\sum_{N \in 2^{\mathbb{Z}}} \langle N\rangle^{2 s}\left|P^{\sharp}_{N} f\right|^{2}\right)^{1 / 2}\right\|_{L_{x}^{p}}
$$
\end{lemma}
A pair $(p, q)$ is called admissible if
$$
\frac{2}{p}=d\left(\frac{1}{2}-\frac{1}{q}\right)
$$
and $4 \leq p \leq \infty$ when $d=1 ; 2<p \leq \infty$ when $d=2$; or $2 \leq p \leq \infty$ when $d \geq 3$. By intertwining relations \eqref{eq1.12} and the classical Strichartz estimate, Suppose $(p, q)$ and $(\tilde{p}, \tilde{q})$ are admissible pairs, and $I \subset \mathbf{R}$ is a possibly infinite time interval. Then we have the following Strichartz estimate for the Schr\"{o}dinger operator $H=-\Delta + V$ providing that $V$ satisfies \(\mathrm{H} 1, \mathrm{H} 2,\) and \(\mathrm{H} 3^{*} \).
\begin{lemma}[\cite{PZS}, Strichartz estimate for the Schr\"{o}dinger operator]
\begin{align}
\left\|e^{i t H} u_{0}\right\|_{L_{t}^{\bar{p}} L_{x}^{\bar{q}}\left(I \times \mathbf{R}^{d}\right)} & \lesssim_{\tilde{p}, \tilde{q}, d}\left\|u_{0}\right\|_{L^{2}\left(\mathbf{R}^{d}\right)}, \\
\left\|\int_{\mathbf{R}} e^{-i t H} F(t) d t\right\|_{L^{2}\left(\mathbf{R}^{d}\right)} & \lesssim_{p, q, d}\|F\|_{L_{t}^{p^{\prime}} L_{x}^{q^{\prime}}\left(I \times \mathbf{R}^{d}\right)}
\end{align}
and
\begin{equation}\label{YEQ3.33}
\left\|\int_{\tau<t, \tau \in I} e^{i(t-\tau) H} F(\tau) d \tau\right\|_{L_{t}^{\tilde{p}} L_{x}^{\tilde{q}}\left(I \times \mathbf{R}^{d}\right)} \lesssim p, q, \tilde{p}, \tilde{q}\|F\|_{L_{t}^{p^{\prime}} L_{x}^{q^{\prime}}\left(I \times \mathbf{R}^{d}\right)}
\end{equation}
\end{lemma}
 Before we begin to prove  Theorem \ref{mthm1.1'}, though, we need the following maximal and square function estimates, which actually are stated by the Lemma 3.3 in \cite{PZS}.

\begin{lemma}[Lemma 3.3 in \cite{PZS}]\label{re1.6}
(a) Suppose that $W$ is an operator that is point-wise bounded by an $L^{p}$ -bounded
positive operator, that is, satisfying the point-wise bound
\[
|W f(x)| \leq C \tilde{W}|f|(x) \quad \text { for all } f \in L^{p}\left(\mathbb{R}^{d}\right), \quad x \in \mathbb{R}^{d}
\]
for some positive operator $\tilde{W}$ that is bounded on $L^{p}\left(\mathbb{R}^{d}\right)$ for $1 \leq p \leq \infty .$ Let $\psi \in C_{0}^{\infty}\left(\mathbb{R}^{d}\right) .$ For each $n \in \mathbb{R}^{d},$ the operators
\[
f \mapsto \sup _{N \in 2^{\mathbb{Z}}}\left|W \mathrm{e}^{2 \pi \mathrm{i} \frac{n \cdot \nabla}{N}} \psi\left(\frac{\nabla}{N}\right) f\right| \quad \text { and } \quad f \mapsto \sup_{\substack{N_{1}, N_{2} \in 2^{\mathbb{Z}}\\ N_{1} \geq N_{2}}}\left|W \mathrm{e}^{2 \pi \mathrm{i} \frac{\mathrm{n\cdot \nabla}}{N_{1}}} \psi\left(\frac{\nabla}{N_{2}}\right) f\right|
\]
are bounded on $L^{p}$ for all $1<p \leq \infty$ with a bound $\lesssim\langle n\rangle^{d}$.

(b) For each $n \in \mathbb{R}^{d}$, the operators
\[
f \mapsto \sup _{N \in 2^{\mathbb{Z}}}\left| \mathrm{e}^{2 \pi \mathrm{i} \frac{n \cdot D^{\sharp}}{N}} \psi\left(\frac{D^{\sharp}}{N}\right) f\right| \quad \text { and } \quad f \mapsto \sup_{\substack{N_{1}, N_{2} \in 2^{\mathbb{Z}}\\ N_{1} \geq N_{2}}}\left| \mathrm{e}^{2 \pi \mathrm{i} \frac{\mathrm{n\cdot D^{\sharp}}}{N_{1}}} \psi\left(\frac{D^{\sharp}}{N_{2}}\right) f\right|
\]
are bounded on $L^{p}\left(\mathbb{R}^{d}\right)$ for all $1<p \leq \infty$ with a bound $\lesssim\langle n\rangle^{d}$.

(c) Let $U$ be any bounded linear operator on $L^{p}$ for some $1 \leq p<\infty$ and suppose that $\left\{f_{n}\right\} \subset L^{p}\left(\mathbb{R}^{d}\right)$ is a sequence of functions. Then
\[
\left\|\left(\sum_{n \in \mathbb{Z}}\left|U f_{n}\right|^{2}\right)^{1 / 2}\right\|_{L^{p}(\mathbb{R}^{d})} \lesssim \left\|\left(\sum_{n \in \mathbb{Z}}\left|f_{n}\right|^{2}\right)^{1 / 2}\right\|_{L^{p}\left(\mathbb{R}^{d}\right)},
\]
whenever the right-hand side is finite.

(d) Moreover, if $\phi$ is smooth and supported on an annulus, the operator
\[
f \mapsto\left(\sum_{N_{2} \in 2^{\mathbb{Z}}} \sup _{\substack{N_1 \in 2^{\mathbb{Z}}\\N_{1} \geq N_{2}}}\left|e^{2 \pi i \frac{n\cdot D^{\sharp}}{N_{1}}} \phi\left(\frac{D^{\sharp}}{N_{2}}\right) f\right|^{2}\right)^{1 / 2}
\]
is bounded on $L^{p}\left(\mathbb{R}^{d}\right)$ for all $1<p<\infty$ with bound $\lesssim\langle n\rangle^{d}$.

(e)  we have the following point-wise inequality
\[
\left|\psi\left(\frac{\nabla}{N}\right) f(x-y)\right| \lesssim\langle N|y|\rangle^{d} M f(x),
\]
where $M f$ is the Hardy-Littlewood maximal function.
\end{lemma}

\section{the proof of  Theorem \ref{mthm1.1'}}

\begin{proof}[Proof of  Theorem \ref{mthm1.1'}] Taking one test function $f_{k+1}\in L^r(\mathbb{R}^d)$, we define $k+1$-linear form $\Lambda(f_1,\cdots,f_k,f_{k+1})$ which is associated to the $k$-linear operator $T(f_1,\ldots,f_k)$  in the framework of the distorted Fourier transform as follows,
\begin{equation}\label{eqy3.1.1}
	\begin{split} &\Lambda(f_1,\cdots,f_k,f_{k+1})\\
        :=&\int_{\mathbb{R}^{d}}T(f_1,\cdots,f_k)f_{k+1}(x)dx\\
		 =&\idotsint m(\xi_1,\cdots,\xi_k)f_1^{\sharp}(\xi_1)\cdots f^{\sharp}_k(\xi_k)e(x,\xi_1)\cdots e(x,\xi_k)f_{k+1}(x)d\xi_1 \cdots d\xi_kdx\\
        =&\idotsint m(\xi_1,\cdots,\xi_k)f_1^{\sharp}(\xi_1)\cdots f^{\sharp}_k(\xi_k)f_{k+1}^{\sharp}(\xi_{k+1})M(\xi_1,\cdots,\xi_k,\xi_{k+1}) d\xi_1 \cdots d\xi_k d\xi_{k+1}
	\end{split}
\end{equation}
where $M(\xi_1,\cdots,\xi_k,\xi_{k+1})=\int_{\mathbb{R}^{d}} e(x,\xi_1)e(x,\xi_2)\cdots e(x,\xi_{k+1})dx$ is nonlinear spectral distribution.
By duality, for $\frac{1}{r}+\frac{1}{r'}=1$,
\begin{align*}
\|T(f_1,\cdots,f_k)\|_{L^{r'}}
  \lesssim \sup_{\|f_{k+1}\|_{L^r}\leq 1}|\Lambda(f_1,\cdots,f_k,f_{k+1})|.
\end{align*}
We generalize the multiplier $m(\xi_1,\cdots,\xi_k)$  as $k+1$ variables multiplier $m(\xi_1,\cdots,\xi_k,\xi_{k+1})$,
\begin{align*}
  &\Lambda(f_1,\cdots,f_k,f_{k+1})\\
  =&\idotsint m(\xi_1,\cdots,\xi_k,\xi_{k+1})f_1^{\sharp}(\xi_1)\cdots f^{\sharp}_k(\xi_k)f_{k+1}^{\sharp}(\xi_{k+1})M(\xi_1,\cdots,\xi_k,\xi_{k+1}) d\xi_1 \cdots d\xi_k d\xi_{k+1}
\end{align*}




{\bf{Step1}}. Decomposition of $\Lambda$. we start by Littlewood-Paley decomposition of $f_j$ with respect to distorted Fourier transform.
$$f_j=\sum_{N_j\in 2^{\mathbb{Z}}}P^{\sharp}_{N_j}f_j=\sum_{N_j\in 2^{\mathbb{Z}}}f_{j;N_j},\,\,\, j=1,\cdots,k+1.$$
As a result, we obtain that
\begin{equation}\label{yeq2.1}
\begin{split}
 \Lambda(f_1,\cdots,f_{k+1})&=\sum_{N_1,\cdots,N_{k+1}\in 2^{\mathbb{Z}}}\Lambda(f_{1;N_1},\cdots,f_{k+1;N_{k+1}})\\
 &=:  \Lambda_L(f_1,\cdots,f_{k+1})+ \Lambda_H(f_1,\cdots,f_{k+1}).
\end{split}
\end{equation}
where $$\Lambda_L(f_1,\cdots,f_{k+1}):=\sum\limits_{N_1\leq 1,\cdots,N_{k+1}\leq 1}\Lambda(f_{1;N_1},\cdots,f_{k+1;N_{k+1}})=\Lambda(f_{1;\leq1},\cdots,f_{k+1;\leq1})$$ and
\begin{equation}
\begin{split}
      \Lambda_H(f_1,\cdots,f_{k+1})&:=\sum_{\max\{N_1,\cdots,N_{k+1}\}\geq 1}\Lambda(f_{1;N_1},\cdots,f_{k+1;N_{k+1}})\\
      &=\sum_{N_1\geq \max\{1,N_2,\cdots,N_{k+1}\}}\Lambda(f_{1;N_1},\cdots,f_{k+1;N_{k+1}})\\
      &\quad+\sum_{N_2\geq \max\{1,N_1,N_3,\cdots,N_{k+1}\}}\Lambda(f_{1;N_1},\cdots,f_{k+1;N_{k+1}})\\
      &\quad+ \cdots \\
      &\quad+\sum_{N_{k+1}\geq \max\{1,N_1,N_3,\cdots,N_{k}\}}\Lambda(f_{1;N_1},\cdots,f_{k+1;N_{k+1}})\\
      &=:I_1+\cdots+I_{k+1}.
\end{split}
\end{equation}
We first take $\Lambda_L(f_1,\cdots,f_{k+1})$ into consideration.  Let $\tilde{\psi} \in C^{\infty}_0(\mathbb{R}^d)$ be given such that $\tilde{\psi}\psi=\psi$. Define $\tilde{m}$ by
 \begin{equation}\label{2222'}
 \tilde{m}\left(\xi_1,\xi_2,\cdots,\xi_{k+1}\right) :=m(\xi_1,\xi_2,\cdots,\xi_{k+1})\tilde{\psi}(\xi_1) \tilde{\psi}(\xi_2)\cdots \tilde{\psi}(\xi_{k+1}),
 \end{equation}
then we expand $\tilde{m}$ in a Fourier series, if $(\xi_1,\cdots,\xi_{k+1})\in [-K/2,K/2]^{(k+1)d},$

\begin{equation}
 \tilde{m}\left(\xi_1,\cdots,\xi_{k+1}\right)=\sum_{n_1,\cdots,n_{k+1}\in \mathbb{Z}^d}a(n_1,\cdots,n_{k+1})e^{\frac{2\pi i}{K}(\sum_{j=1}^{k+1}n_j\cdot\xi_j)}.
\end{equation}
Using the stationary phase method, we have the bound for $a(n_1,\cdots,n_{k+1})$:
\begin{equation}
  \left|a(n_1,\cdots,n_{k+1}) \right|\lesssim (1+|n_1|+\cdots+|n_{k+1}|)^{-3(k+1)d},
\end{equation}
thus by H\"{o}lder's inequality and Lemma \ref{re1.6} (b), we have
\begin{equation*}
\begin{split}
&|\Lambda_L(f_1,\cdots,f_{k+1})|\\
=&|\Lambda(f_{1;\leq1},\cdots,f_{k+1;\leq1})|\\
=& |\sum_{n_1,n_2,\cdots,n_{k+1}}a(n_1,\cdots,n_{k+1}) \int f_{1;\leq 1,n_1}(x) f_{2;\leq 1,n_2}(x)\cdots f_{k+1;\leq 1,n_{k+1}}(x) dx|\\
\lesssim &  \sum_{n_1,n_2,\cdots,n_{k+1}} (1+|n_1|+\cdots+|n_{k+1}|)^{-3(k+1)d}\prod_{j=1}^{k+1}\|f_{j;\leq1,n_j}\|_{L^{p_j}}\\
\lesssim & \prod_{j=1}^{k+1}\|f_{j}\|_{L^{p_j}}
\end{split}
\end{equation*}where $f_{j;\leq1,n_j}(x):=\mathcal{F}^{\sharp^{-1}}(e^{\frac{2\pi i}{K}n_j\cdot\xi_j}f^{\sharp}_{j;\leq1}(\xi_j))$.

For $\Lambda_H(f_1,\cdots,f_{k+1})$ part,  we can just treat with $I_1$, because the other terms can be controlled in the same way,
\begin{equation}
\begin{split}
      I_1&=\sum_{N_1\geq \max\{1,N_2,\cdots,N_{k+1}\}}\Lambda(f_{1;N_1},\cdots,f_{k+1;N_{k+1}})\\
      &=\sum_{N_1\geq1}\sum_{N_2\leq N_1}\sum_{\substack{N_3,N_4,\cdots,N_{k+1}\\ (\leq N_2)}} \Lambda(f_{1;N_1},\cdots,f_{k+1;N_{k+1}})\\
      &\quad+\sum_{N_1\geq1}\sum_{N_3\leq N_1}\sum_{\substack{N_2,N_4,\cdots,N_{k+1}\\ (\leq N_2)}} \Lambda(f_{1;N_1},\cdots,f_{k+1;N_{k+1}})\\
      &\quad+ \cdots \\
      &\quad+\sum_{N_1\geq1}\sum_{N_{k+1}\leq N_1}\sum_{\substack{N_2,N_3,\cdots,N_{k}\\ (\leq N_2)}} \Lambda(f_{1;N_1},\cdots,f_{k+1;N_{k+1}})\\
      &=:I_{1,2}+\cdots+I_{1,,k+1}.
\end{split}
\end{equation}
In the following, due to the similarities, we still only estimate the first term $I_{1,2}$, and our default summation range  about $N_1$ is  $N_1\geq1$, if not necessary, we will not mention it again.

 Let $\tilde{\phi}\in C_0^{\infty}(\mathbb{R})$ be given such that $\tilde{\phi}\phi=\phi$. Define $\tilde{m}^{N_1}$ by
 \begin{equation}\label{2222}
 \tilde{m}^{N_1}\left(\frac{\xi_1}{N_1},\frac{\xi_2}{N_1},\cdots,\frac{\xi_{k+1}}{N_1}\right) :=m(\xi_1,\xi_2,\cdots,\xi_{k+1})\tilde{\phi}\left(\frac{\xi_1}{N_1}\right) \tilde{\phi}\left(\frac{\xi_2}{N_1}\right)\cdots \tilde{\phi}\left(\frac{\xi_{k+1}}{N_1}\right),
 \end{equation}
then we expand $\tilde{m}^{N_1}$ in a Fourier series, if $(\xi_1,\cdots,\xi_{k+1})\in [-K/2,K/2]^{(k+1)d},$

\begin{equation}
 \tilde{m}^{N_1}\left(\xi_1,\cdots,\xi_{k+1}\right)=\sum_{n_1,\cdots,n_{k+1}\in \mathbb{Z}^d}a^{N_1}(n_1,\cdots,n_{k+1})e^{\frac{2\pi i}{K}(\sum_{j=1}^{k+1}n_j\cdot\xi_j)}.
\end{equation}
Using the stationary phase method, we have the bound for $a^{N_1}(n_1,\cdots,n_{k+1})$:

\begin{equation}
  \left|a^{N_1}(n_1,\cdots,n_{k+1}) \right|\lesssim (1+|n_1|+\cdots+|n_{k+1}|)^{-3(k+1)d}
\end{equation}
meanwhile,
\begin{equation}
\begin{split}
  I_{1,2} &=\sum_{N_1}\sum_{N_2\leq N_1}\sum_{\substack{N_3,N_4,\cdots,N_{k+1}\\ (\leq N_2)}} \Lambda(f_{1;N_1},\cdots,f_{k+1;N_{k+1}})\\
    &=\sum_{N_1}\sum_{N_2\leq N_1}\sum_{\substack{N_3,N_4,\cdots,N_{k+1}\\ (\leq N_2)}}\sum_{n_1,n_2,\cdots,n_{k+1}}a^{N_1}(n_1,\cdots,n_{k+1}) \Xi^{n_1,\cdots,n_{k+1}}_{N_1,\cdots,N_{k+1}}.
\end{split}
\end{equation}
where
\begin{equation*}
 \begin{split}
  &\Xi^{n_1,\cdots,n_{k+1}}_{N_1,\cdots,N_{k+1}}\\
  :=&\idotsint \left(e^{\frac{2\pi i}{KN_1}n_1\cdot\xi_1}f^{\sharp}_{1;N_1}(\xi_1)\right)\cdots \left(e^{\frac{2\pi i}{KN_1}n_{k+1}\cdot\xi_{k+1}}f^{\sharp}_{k+1;N_{k+1}}(\xi_{k+1})\right)\\
  & \qquad M(\xi_1,\cdots,\xi_{k+1}) d\xi_1\cdots d\xi_{k+1}\\
  =& \idotsint f^{\sharp}_{1;N_1,n_1}(\xi_1) f^{\sharp}_{2;N_2,n_2,N_1}(\xi_2)\cdots f^{\sharp}_{k+1;N_{k+1},n_{k+1},N_1}(\xi_{k+1})\\
  & \qquad M(\xi_1,\cdots,\xi_{k+1}) d\xi_1\cdots d\xi_{k+1}\\
  =& \int f_{1;N_1,n_1}(x) f_{2;N_2,n_2,N_1}(x)\cdots f_{k+1;N_{k+1},n_{k+1},N_1}(x)dx
 \end{split}
\end{equation*}
Here we have denoted
$f^{\sharp}_{1;N_1,n_1}(\xi_1):=e^{\frac{2\pi i}{KN_1}n_1\cdot\xi_1}f^{\sharp}_{1;N_1}(\xi_1)$, and  $f^{\sharp}_{j;N_j,n_j,N_1}(\xi_j):=e^{\frac{2\pi i}{KN_1}n_j\cdot\xi_j}f^{\sharp}_{j;N_j}(\xi_j)$, $j=2,\cdots,k+1.$
Therefore,
\begin{equation}\label{111111}
\begin{split}
 |I_{1,2}| =&\left|\sum_{N_1}\sum_{N_2\leq N_1}\sum_{\substack{N_3,N_4,\cdots,N_{k+1}\\ (\leq N_2)}} \Lambda(f_{1;N_1},\cdots,f_{k+1;N_{k+1}})\right|\\
 =&\left|\sum_{N_1}\sum_{N_2\leq N_1}\sum_{\substack{N_3,N_4,\cdots,N_{k+1}\\ (\leq N_2)}} \sum_{n_1,n_2,\cdots,n_{k+1}}a^{N_1}(n_1,\cdots,n_{k+1}) \Xi^{n_1,\cdots,n_{k+1}}_{N_1,\cdots,N_{k+1}}\right|\\
  =&\left|\sum_{N_1}\sum_{n_1,\cdots,n_{k+1}}a^{N_1}(n_1,\cdots,n_{k+1}) \sum_{N_2\leq N_1}\sum_{\substack{N_3,N_4,\cdots,N_{k+1}\\ (\leq N_2)}} \Xi^{n_1,\cdots,n_{k+1}}_{N_1,\cdots,N_{k+1}}\right|\\
 \lesssim &\sum_{n_1,\cdots,n_{k+1}}\left(1+|n_1|+\cdots+|n_{k+1}|\right)^{-3(k+1)d}\sum_{N_1}\left| \sum_{\substack{N_2 \\(\leq N_1)}}\sum_{\substack{N_3,N_4,\cdots,N_{k+1}\\ (\leq N_2)}} \Xi^{n_1,\cdots,n_{k+1}}_{N_1,\cdots,N_{k+1}} \right|.
\end{split}
\end{equation}
As a result, we are reduced to proving the following estimates:
\begin{enumerate}
  \item For $p_j\in(0,\infty)$ s.t. $\sum_{j=1}^{k+1}\frac{1}{p_j}=1$, assume that the Riesz transform \(\Re=\nabla(-\Delta+V)^{-1 / 2}\) is bounded on $L^{p_j}, j=1,\ldots,k+1$, then we have
\begin{equation}\label{1.2.1}
 \begin{split}
  \sum\limits_{N_1\geq 1}\left| \sum_{\substack{N_2 \\(\leq N_1)}}\sum_{\substack{N_3,N_4,\cdots,N_{k+1}\\ (\leq N_2)}} \Xi^{n_1,\cdots,n_{k+1}}_{N_1,\cdots,N_{k+1}} \right| \lesssim \prod_{j=1}^{k+1} \langle n_j\rangle^d \|f_j\|_{L^{p_j}}.
 \end{split}
\end{equation}
  \item Suppose instead that $V$ satisfies \(\mathrm{H} 3^{*} \) and we do not have the assumption of Riesz transform \(\Re=\nabla(-\Delta+V)^{-1 / 2}\) being bounded on $L^{p_j}, j=1,\ldots,k+1$ any more, then for $p_j, \bar{p}_j\in(0,\infty)$ s.t. $\sum_{j=1}^{k+1}\frac{1}{p_j}=1, \sum_{j=1}^{k+1}\frac{1}{\tilde{p}_j}=1+\frac{\epsilon}{d}$, we have
\begin{equation}\label{1.2.2}
 \begin{split}
  \sum\limits_{N_1\geq 1}\left| \sum_{\substack{N_2 \\(\leq N_1)}}\sum_{\substack{N_3,N_4,\cdots,N_{k+1}\\ (\leq N_2)}} \Xi^{n_1,\cdots,n_{k+1}}_{N_1,\cdots,N_{k+1}} \right|\lesssim \prod_{j=1}^{k+1} \langle n_j\rangle^d \|f_j\|_{L^{p_j}}+\prod_{j=1}^{k+1} \langle n_j\rangle^d\|f_j\|_{L^{\tilde{p}_j}}.
 \end{split}
\end{equation}
\end{enumerate}
{\bf{Step 2}}. Recall  $$M(\xi_1,\cdots,\xi_k,\xi_{k+1})=\int_{\mathbb{R}^{d}} e(x,\xi_1)e(x,\xi_2)\cdots e(x,\xi_{k+1})dx.$$
Using Green's formula and the definition of distorted plane wave functions, formally we have
\begin{align*}
&|\xi_1|^2M(\xi_1,\cdots,\xi_k,\xi_{k+1})\\
&=\int_{\mathbb{R}^{d}} He(x,\xi_1)e(x,\xi_2)\cdots e(x,\xi_{k+1})dx\\
&=\int_{\mathbb{R}^{d}} V(x)e(x,\xi_1)e(x,\xi_2)\cdots e(x,\xi_{k+1})dx \\
&\qquad - \int_{\mathbb{R}^{d}}  e(x,\xi_1)\Delta[e(x,\xi_2)\cdots e(x,\xi_{k+1})]dx\\
&=\sum_{j=2}^{k+1}|\xi_j|^2 M(\xi_1,\cdots,\xi_k,\xi_{k+1})\\
&-(k-1)\int_{\mathbb{R}^{d}}V(x) e(x,\xi_1)e(x,\xi_2)\cdots e(x,\xi_{k+1})dx\\
&+2\sum_{2\leq j<l\leq k+1}\int_{\mathbb{R}^{d}} e(x,\xi_1)e(x,\xi_2)\cdots\nabla e(x,\xi_j)\cdot \nabla e(x,\xi_l)\cdots e(x,\xi_{k+1})dx
\end{align*}
which holds in the sense of distribution.
Thus,
\begin{align*}
 &\Xi^{n_1,\cdots,n_{k+1}}_{N_1,\cdots,N_{k+1}}\\
  =& \idotsint f^{\sharp}_{1;N_1,n_1}(\xi_1) f^{\sharp}_{2;N_2,n_2,N_1}(\xi_2)\cdots f^{\sharp}_{k+1;N_{k+1},n_{k+1},N_1}(\xi_{k+1})\\
   &\times M(\xi_1,\cdots,\xi_{k+1}) d\xi_1\cdots d\xi_{k+1}\\
   =& \sum_{j=2}^{k+1}\idotsint\frac{|\xi_j|^2}{|\xi_1|^2}f^{\sharp}_{1;N_1,n_1}(\xi_1) f^{\sharp}_{2;N_2,n_2,N_1}(\xi_2)\cdots f^{\sharp}_{k+1;N_{k+1},n_{k+1},N_1}(\xi_{k+1})\\
    &\times M(\xi_1,\cdots,\xi_{k+1}) d\xi_1\cdots d\xi_{k+1}\\
  &-(k-1)\frac{1}{N_1^2}\int \underline{f}_{1;N_1,n_1}(x) f_{2;N_2,n_2,N_1}(x)\cdots f_{k+1;N_{k+1},n_{k+1},N_1}(x)V(x)dx\\
  &+2\sum_{2\leq j<l\leq k+1}\frac{1}{N_1^2}\int \underline{f}_{1;N_1,n_1}(x) f_{2;N_2,n_2,N_1}(x)\\
  &\quad\times \cdots\times\nabla f_{j;N_j,n_j,N_1}(x)\cdot \nabla f_{l;N_l,n_l,N_1}(x)\cdots f_{k+1;N_{k+1},n_{k+1},N_1}(x)dx\\
  &=:I+II+III.
\end{align*}
Here we denote
$\underline{f}_{1;N_{1}, n_{1}}:=\mathcal{F}^{\sharp^{-1}} \frac{N_{1}^{2}}{|\xi_1|^{2}} \mathcal{F}^{\sharp} f_{1;N_{1}, n_{1}}=\mathcal{F}^{\sharp^{-1}} \mathrm{e}^{2 \pi \mathrm{i} \frac{n_{1} \xi_1}{K N_{1}}} \underline{\phi}\left(\frac{\xi_1}{N_{1}}\right) \mathcal{F}^{\sharp} f_1, \, \text { with } \underline{\phi}(\xi):=|\xi|^{-2} \phi(\xi).$

We start with the contribution of $I$: Writing $f_{l;\leq N_j,n_l,N_1}=\sum\limits_{N_l\leq N_j}f_{l; N_l,n_l,N_1},$ $2\leq j\leq l\leq k+1,$ and
$
\tilde{f}_{j;N_{j}, n_{j}, N_{1}}:=\mathcal{F}^{\sharp^{-1}}  \frac{|\xi_j|^{2}}{N_{j}^{2}} \mathcal{F}^{\sharp} f_{j;N_{j}, n_{j}, N_{1}}=\mathcal{F}^{\sharp^{-1}} \mathrm{e}^{2 \pi \mathrm{i} \frac{n_{j} \xi_j}{K N_{1}}} \tilde{\phi}\left(\frac{\xi_j}{N_{j}}\right) \mathcal{F}^{\sharp} f_j, \,j=2,\cdots,k+1,\, \text { with } \tilde{\phi}(\xi):=|\xi|^{2} \phi(\xi),
$
 By H\"{o}lder's inequality and Lemma \ref{re1.6}(e), (b) and (d), for $1=\sum\limits_{l=1}^{k+1} \frac{1}{p_l}$, we have

\begin{equation}\label{eq1.3.5'}
\begin{split}
&\sum\limits_{N_1}\left| \sum_{N_2\leq N_1}\sum_{\substack{N_3,N_4,\cdots,N_{k+1}\\ (\leq N_2)}} I \right|\\
\leq&\sum_{j=2}^{k+1}\sum\limits_{N_1}\Big| \sum_{N_2\leq N_1}\sum_{\substack{N_3,N_4,\cdots,N_{k+1}\\ (\leq N_2)}} \int\frac{|\xi_j|^2}{|\xi_1|^2}f^{\sharp}_{1;N_1,n_1}(\xi_1) f^{\sharp}_{2;N_2,n_2,N_1}(\xi_2)\\
&\cdots f^{\sharp}_{k+1;N_{k+1},n_{k+1},N_1}(\xi_{k+1}) M(\xi_1,\cdots,\xi_{k+1}) d\xi_1\cdots d\xi_{k+1} \Big|\\
=&\sum_{j=2}^{k+1}\sum\limits_{N_1}\Big| \sum_{N_2\leq N_1}\sum_{N_j\leq N_2} \frac{N_j^2}{N_1^2} \int  \underline{f}^{\sharp}_{1;N_1,n_1}(\xi_1)f^{\sharp}_{2;N_2,n_2,N_1}(\xi_2) f^{\sharp}_{3;\leq N_2,n_3,N_1}(\xi_3)\\
&\cdots f^{\sharp}_{j-1;\leq N_2,n_{j-1},N_1}(\xi_{j-1}) \tilde{f}^{\sharp}_{j;N_j,n_j,N_1}(\xi_j) \\
&\quad \times f^{\sharp}_{j+1;\leq N_{2},n_{j+1},N_1}(\xi_{j+1}) \cdots f^{\sharp}_{k+1;\leq N_{2},n_{k+1},N_1}(\xi_{k+1}) M(\xi_1,\cdots,\xi_{k+1}) d\xi_1\cdots d\xi_{k+1} \Big|\\
=&\sum_{j=2}^{k+1}\sum\limits_{N_1}\Big| \sum_{N_2\leq N_1}\sum_{N_j\leq N_2} \frac{N_j^2}{N_1^2} \int  \underline{f}_{1;N_1,n_1}(x) f_{2;N_2,n_2,N_1}(x) f_{3;\leq N_2,n_3,N_1}(x) \\
&\cdots f_{j-1;\leq N_2,n_{j-1},N_1}(x) \tilde{f}_{j;N_j,n_j,N_1}(x) f_{j+1;\leq N_{2},n_{j+1},N_1}(x) \cdots f_{k+1;\leq N_{2},n_{k+1},N_1}(x) dx \Big|\\
\leq& \sum_{j=2}^{k+1}\int\sum\limits_{\substack {N_1,N_2\\(N_2\leq N_1)}}  \frac{N_2^2}{N_1^2} |\underline{f}_{1;N_1,n_1}(x)| |f_{2;N_2,n_2,N_1}(x)| |f_{3;\leq N_2,n_3,N_1}(x)|\cdots |f_{j-1;\leq N_2,n_{j-1},N_1}(x)|   \\
&\times \sup_{N_j\leq N_2}|\tilde{f}_{j;N_j,n_j,N_1}(x)| |f_{j+1;\leq N_{2},n_{j+1},N_1}(x)| \cdots |f_{k+1;\leq N_{2},n_{k+1},N_1}(x)| dx   \\
  \lesssim& \Big\|\Big(\sum_{N_{1}}   \left|\underline{f}_{1; N_{1},n_{1}}  \right|^2\Big)^{1/2}\Big\|_{L^{p_{1}}} \Big\| \Big(\sum_{N_{2}}  \sup_{\substack{N_1\\(N_{2}\leq N_1)}} \left|f_{2; N_{2},n_{2},N_1}  \right|^2\Big)^{1/2}\Big\|_{L^{p_{2}}}  \\
  & \times \Big\|\sup_{\substack{N_j, N_1 \\ (N_j\leq N_1)}}\left|\tilde{f}_{j; N_j,n_j,N_1} \right|\Big\|_{L^{p_j}} \prod^{k+1}_{l=3,l\neq j}\Big\|\sup_{\substack{N_2, N_1 \\ (N_2\leq N_1)}}\left|f_{l; \leq N_2,n_l,N_1} \right|\Big\|_{L^{p_l}}\\
   \lesssim &  \prod_{l=1}^{k+1}\langle n_l\rangle^d \left\| f_l \right\|_{L^{p_l}}.
\end{split}
\end{equation}


For $II$, we use the Sobolev embedding in the distorted Fourier transform setting. Let $s_j>0$ satisfy $\sum_{j=1}^{2}s_j<1, s_j\leq \frac{d}{p_j}$, and we denote $\frac{1}{q_j}=\frac{1}{p_j}-\frac{s_j}{d}$, $j=1,2$.  By H\"{o}lder inequality, Lemma \ref{re1.6}(e), (b), (d) and Sobolev embedding, we have
\begin{equation}\label{req1.3.5'}
\begin{split}
&\sum\limits_{N_1\geq1}\left| \sum\limits_{N_2\leq N_1} \sum\limits_{\substack{N_3,\cdots,N_{k+1}\\  (\leq N_2)}} II \right|\\
   \lesssim & \sum\limits_{N_1\geq1}\left| \sum\limits_{N_2\leq N_1} \sum\limits_{\substack{N_3,\cdots,N_{k+1}\\  (\leq N_2)}}\frac{1}{N_1^2}\int \underline{f}_{1;N_1,n_1}(x) f_{2;N_2,n_2,N_1}(x)\cdots f_{k+1;N_{k+1},n_{k+1},N_1}(x)V(x)dx \right| \\
      \lesssim& \sum\limits_{\substack{N_1,N_2\\  ( N_2\leq N_1,N_1\geq1)}}\frac{N_2^{s_2}}{N_1^{2-s_1}} \left|\int \left(N_1^{-s_1}\underline{f}_{1;N_1,n_1}\right) \left(N_2^{-s_2}f_{2;N_2,n_2,N_1}\right) \left(\prod_{j=3}^{k+1}f_{j;\leq N_2,n_j,N_1}\right)V(x)dx\right|\\
      \lesssim & \int \left(\sum_{N_1\geq1}N_1^{-2s_1}\left|\underline{f}_{1;N_{1}, n_{1}}  \right|^2\right)^{1/2} \left(\sum_{N_2}N_2^{-2s_2}\sup_{\substack{N_1\geq1\\(N_1\geq N_2)}}\left|f_{2;N_{2}, n_{2}, N_{1}} \right|^2\right)^{1/2}\\
   &\quad\times \prod_{j=3}^{k+1}\sup_{\substack{N_1,N_2\\(N_2\leq N_1)}}\left|f_{j; \leq N_2,n_{j},N_1} \right| |V(x)|dx \\
   \lesssim &\|V\|_{L^{\frac{d}{s_1+s_2}}}\prod_{j=1}^{2}\langle n_j\rangle^d \left\|f_j\right\|_{\dot{W}^{-s_j,q_j}_{\sharp}} \prod_{l=3}^{k+1}\langle n_l\rangle^d \left\| f_l \right\|_{L^{p_l}}\\
   \lesssim & \prod_{j=1}^{k+1}\langle n_j\rangle^d \left\| f_j \right\|_{L^{p_j}}
\end{split}
\end{equation}

Finally we estimate the contribution of the term $III$, which is also the one that causes us to complete the proofs separately according to the assumption that the Riesz transform is bounded or unbounded. we first bound the contribution of III under the assumption that $\mathfrak{R}$  is bounded on $L^{p_j} (\mathbb {R}^d),\,j=1,\cdots,k+1$.
\begin{equation*}
\begin{split}
&\sum\limits_{N_1\geq 1}\left| \sum_{N_2\leq N_1}\sum_{\substack{N_3,N_4,\cdots,N_{k+1}\\ (\leq N_2, N_l\leq N_j)}} III \right|\\
\lesssim& \int \sum_{ 1\leq N_1} \sum_{N_2\leq N_1}  \sum_{\substack{ N_l\leq N_j\\(\leq N_2)}} \frac{N_l N_j}{N_1^2} |\underline{f}_{1;N_1,n_1}| |f_{2; N_2,n_2,N_1}|  \left| \mathfrak{R} \tilde{f}_{j;N_j,n_j,N_1}\right| \left|\mathfrak{R} \tilde{f}_{l;N_l,n_l,N_1}\right| \\
&\quad\times\prod^{k+1}_{m=3,m\neq j,l} |f_{m; \leq N_2,n_m,N_1}| dx\\
\end{split}
\end{equation*}
\begin{align*}
       \lesssim & \int  \left(\sum_{N_2} \sup_{N_1(\geq N_2)}\left|f_{2;N_{2}, n_{2}, N_{1}} \right|^2\right)^{1/2} \left(\sum_{N_1\geq 1}\left|\underline{f}_{1;N_{1}, n_{1}} \right|^2\right)^{1/2}  \\
       & \times \prod_{m=3,m\neq l, m\neq j}^{k+1}\sup_{N_2\leq N_1}|f_{m; \leq N_2,n_m,N_1}| \sup_{N_1\geq N_l}\left|\mathfrak{R} \tilde{f}_{l;N_l,n_l,N_1}\right| \sup_{N_1\geq N_j} \left| \mathfrak{R} \tilde{f}_{j;N_j,n_j,N_1}\right|dx\\
 \lesssim & \prod_{m=1}^{k+1}\langle n_m\rangle^d \left\| f_m \right\|_{L^{p_m}}.
\end{align*}
This finishes the proof of \eqref{1.2.1}.

{\bf{Step 3}}. proof of \eqref{1.2.2}. We assume that  the potential $V$ satisfies assumption $H 3^{*},$ which implies the $L^{p}\left(\mathbb{R}^{d}\right)$ boundedness of the operator $\mathfrak{B}: f \mapsto \nabla(I-\Delta+V)^{-1 / 2} f=\nabla\left\langle D^{\sharp}\right\rangle^{-1} f$. This follows directly by noting that $\left\langle D^{\sharp}\right\rangle^{-1}=$ $\Omega\langle\nabla\rangle^{-1} \Omega^{*},$ using assumption $H 3^{*}$ and the boundedness of $\nabla\langle\nabla\rangle^{-1}$. We denote $\tilde{f}_{j;N_j,n_j,N_1}:=\frac{\langle D^{\sharp}\rangle}{\langle N_j\rangle} f_{j;N_j,n_j,N_1}$. We split the analysis into three cases, depending on the size of $N_l$ and $N_j$, by symmetry, we may further assume  $N_l\leq N_j$.

Case 1. $N_l\leq N_j<1$.  In this case, applying Lemma \ref{re1.6}(e), (b), (d), $l^2 \subseteq l^{\infty}$ and Sobolev embedding, for $\sum_{j=1}^{k+1}\frac{1}{\tilde{p}_j}=1+\frac{\epsilon}{d}$, and $2<j\leq l<k+1$,  we bound as follows,

\begin{equation*}
\begin{split}
&\sum\limits_{N_1\geq 1}\left| \sum_{N_2\leq N_1}\sum_{\substack{N_3,N_4,\cdots,N_{k+1}\\ (\leq N_2, N_l\leq N_j\leq1 )}} III \right|\\
\lesssim& \int \sum_{ 1\leq N_1}\frac{1}{N_1^{2-\epsilon}}   \sum_{N_2\leq N_1}\sum_{\substack{N_l\leq N_j\\ \leq \min(N_2,1)}} \frac{N_l^{\epsilon}}{N_1^{\epsilon}} |\underline{f}_{1;N_1,n_1}| |f_{2; N_2,n_2,N_1}|  \left| \mathfrak{B} \tilde{f}_{j;N_j,n_j,N_1}\right| \\
 &\quad \times N_l^{-\epsilon}\left|\mathfrak{B} \tilde{f}_{l;N_l,n_l,N_1}\right| \prod^{k+1}_{m=3,m\neq j,l} |f_{m; \leq N_2,n_m,N_1}| dx\\
 \lesssim & \int  \left(\sum_{N_2} \sup_{N_1(\geq N_2)}\left|f_{2;N_{2}, n_{2}, N_{1}} \right|^2\right)^{1/2} \left(\sum_{N_1\geq 1}\left|\underline{f}_{1;N_{1}, n_{1}} \right|^2\right)^{1/2} \\
       & \times \prod_{m=3,m\neq l, m\neq j}^k\sup_{N_2\leq N_1}|f_{m; \leq N_2,n_m,N_1}| \sup_{N_1\geq N_l}\left|(N_{l})^{-\epsilon}\mathfrak{B} \tilde{f}_{l;N_l,n_l,N_1}\right| \sup_{N_1\geq N_j} \left| \mathfrak{B} \tilde{f}_{j;N_j,n_j,N_1}\right|dx\\
\lesssim & \prod_{m=1,m\neq l}^{k+1}\langle n_m\rangle^d \left\| f_m \right\|_{L^{p_m}}\times \langle n_{l}\rangle^d \left\| f_{l} \right\|_{\dot{W}^{-\epsilon,p_{l}}}\\
 \lesssim & \prod_{m=1}^{k+1}\langle n_m\rangle^d \left\| f_m \right\|_{L^{\tilde{p}_m}}.
\end{split}
\end{equation*}
       

Case 2.  $N_l\leq 1\leq N_j$. Similarly to Case 1, for $\sum_{j=1}^{k+1}\frac{1}{\tilde{p}_j}=1+\frac{\epsilon}{d}$, we can obtain

\begin{equation*}
\begin{split}
&\sum\limits_{N_1\geq 1}\left| \sum_{N_2\leq N_1}\sum_{\substack{N_3,N_4,\cdots,N_{k+1}\\ (\leq N_2,N_l\leq1\leq N_j)}} III \right|\\
\lesssim& \int \sum_{ 1\leq N_1}\frac{1}{N_1^{1-\epsilon}}   \sum_{N_2\leq N_1}\sum_{\substack{N_l\leq1\leq N_j\\(\leq N_2)}} \frac{N_j N_l^{\epsilon}}{N_1^{1+\epsilon}} |\underline{f}_{1;N_1,n_1}| |f_{2; N_2,n_2,N_1}|  \left| \mathfrak{B} \tilde{f}_{j;N_j,n_j,N_1}\right| \\
 &\quad \times N_l^{-\epsilon}\left|\mathfrak{B} \tilde{f}_{l;N_l,n_l,N_1}\right| \prod^{k+1}_{m=3,m\neq j,l} |f_{m; \leq N_2,n_m,N_1}| dx\\
       \lesssim & \int  \left(\sum_{N_2} \sup_{N_1(\geq N_2)}\left|f_{2;N_{2}, n_{2}, N_{1}} \right|^2\right)^{1/2} \left(\sum_{N_1\geq 1}\left|\underline{f}_{1;N_{1}, n_{1}} \right|^2\right)^{1/2}  \\
       & \times \prod_{m=3,m\neq l, m\neq j}^{k+1}\sup_{N_2\leq N_1}|f_{m; \leq N_2,n_m,N_1}| \sup_{N_1\geq N_l}\left|(N_{l})^{-\epsilon}\mathfrak{B} \tilde{f}_{l;N_l,n_l,N_1}\right| \sup_{N_1\geq N_j} \left| \mathfrak{B} \tilde{f}_{j;N_j,n_j,N_1}\right|dx\\
\lesssim & \prod_{m=1,m\neq l}^{k+1}\langle n_m\rangle^d \left\| f_m \right\|_{L^{p_m}}\times \langle n_{l}\rangle^d \left\| f_{l} \right\|_{\dot{W}^{-\epsilon,p_{l}}}\\
 \lesssim & \prod_{m=1}^{k+1}\langle n_m\rangle^d \left\| f_m \right\|_{L^{\tilde{p}_m}}.
\end{split}
\end{equation*}


Case 3. $1\leq N_l\leq N_j$. Similarly, for $\sum_{j=1}^{k+1}\frac{1}{p_j}=1$, we have
\begin{equation*}
\begin{split}
&\sum\limits_{N_1\geq 1}\left| \sum_{N_2\leq N_1}\sum_{\substack{N_3,N_4,\cdots,N_{k+1}\\ (\leq N_2,1\leq N_l\leq N_j)}} III \right|\\
\lesssim& \int \sum_{ 1\leq N_1} \sum_{N_2\leq N_1}  \sum_{\substack{1\leq N_l\leq N_j\\(\leq N_2)}} \frac{N_l N_j}{N_1^2} |\underline{f}_{1;N_1,n_1}| |f_{2; N_2,n_2,N_1}|  \left| \mathfrak{B} \tilde{f}_{j;N_j,n_j,N_1}\right| \left|\mathfrak{B} \tilde{f}_{l;N_l,n_l,N_1}\right|\\
\end{split}
\end{equation*}
\begin{align*}
&\quad\times\prod^{k+1}_{m=3,m\neq j,l} |f_{m; \leq N_2,n_m,N_1}| dx\\
       \lesssim & \int  \left(\sum_{N_2} \sup_{N_1(\geq N_2)}\left|f_{2;N_{2}, n_{2}, N_{1}} \right|^2\right)^{1/2} \left(\sum_{N_1\geq 1}\left|\underline{f}_{1;N_{1}, n_{1}} \right|^2\right)^{1/2}  \\
       & \times \prod_{m=3,m\neq l, m\neq j}^{k+1}\sup_{N_2\leq N_1}|f_{m; \leq N_2,n_m,N_1}| \sup_{N_1\geq N_l}\left|\mathfrak{B} \tilde{f}_{l;N_l,n_l,N_1}\right| \sup_{N_1\geq N_j} \left| \mathfrak{B} \tilde{f}_{j;N_j,n_j,N_1}\right|dx\\
   \lesssim & \prod_{m=1}^{k+1}\langle n_m\rangle^d \left\| f_m \right\|_{L^{p_m}}.
\end{align*}

\end{proof}

\section{application}
\subsection{Application 1: Leibniz's law of integer order derivations} Taking one test function $f_{k+1}\in L^r(\mathbb{R}^d)$, for $s\geq0$, since the operator $H=-\nabla +V$ is self-adjoint, and Plancherel's theorem still holds for distorted Fourier transform,   we can get   
\begin{equation*}
	\begin{split} &\int_{\mathbb{R}^{d}}H^s T(f_1,\cdots,f_k)f_{k+1}(x)dx\\
        =&\int_{\mathbb{R}^{d}}T(f_1,\cdots,f_k)H^s f_{k+1}(x)dx\\
		 =&\idotsint m(\xi_1,\cdots,\xi_k)f_1^{\sharp}(\xi_1)\cdots f^{\sharp}_k(\xi_k)e(x,\xi_1)\cdots e(x,\xi_k)H^s f_{k+1}(x)d\xi_1 \cdots d\xi_kdx\\
        =&\idotsint m(\xi_1,\cdots,\xi_k)f_1^{\sharp}(\xi_1)\cdots f^{\sharp}_k(\xi_k)f_{k+1}^{\sharp}(\xi_{k+1})\\
        &\quad\times|\xi_{k+1}|^{2s}M(\xi_1,\cdots,\xi_k,\xi_{k+1}) d\xi_1 \cdots d\xi_k d\xi_{k+1}
	\end{split}
\end{equation*}
where $M(\xi_1,\cdots,\xi_k,\xi_{k+1})=\int_{\mathbb{R}^{d}} e(x,\xi_1)e(x,\xi_2)\cdots e(x,\xi_{k+1})dx$ is nonlinear spectral distribution. Let $s=1$ and $s=\frac{1}{2}$ respectively for example,
using Green's formula and the definition of distorted plane wave functions, formally in the sense of distribution, we have

{\bf{Case $s=1$}}:
\begin{align*}
&|\xi_{k+1}|^2M(\xi_1,\cdots,\xi_k,\xi_{k+1})\\
&=\int_{\mathbb{R}^{d}} e(x,\xi_1)e(x,\xi_2)\cdots H e(x,\xi_{k+1})dx\\
&=\int_{\mathbb{R}^{d}} V(x)e(x,\xi_1)e(x,\xi_2)\cdots e(x,\xi_{k+1})dx \\
&\qquad - \int_{\mathbb{R}^{d}}  e(x,\xi_{k+1})\Delta[e(x,\xi_1)\cdots e(x,\xi_{k})]dx\\
&=\sum_{j=1}^{k}|\xi_j|^2 M(\xi_1,\cdots,\xi_k,\xi_{k+1})\\
&-(k-1)\int_{\mathbb{R}^{d}}V(x) e(x,\xi_1)e(x,\xi_2)\cdots e(x,\xi_{k+1})dx\\
&+2\sum_{1\leq j<l\leq k}\int_{\mathbb{R}^{d}} e(x,\xi_1)e(x,\xi_2)\cdots\nabla e(x,\xi_j)\cdot \nabla e(x,\xi_l)\cdots e(x,\xi_{k+1})dx
\end{align*}
Therefore,
\begin{equation}\label{y3.1}
	\begin{split} &\int_{\mathbb{R}^{d}}H T(f_1,\cdots,f_k)f_{k+1}(x)dx\\
        =&\int_{\mathbb{R}^{d}}T(H f_1,\cdots,f_k)f_{k+1}(x)dx+\cdots+\int_{\mathbb{R}^{d}}T( f_1,\cdots, H f_k) f_{k+1}(x)dx\\
        &\quad -(k-1) \int_{\mathbb{R}^{d}}T( f_1,\cdots,f_k)V(x)f_{k+1}(x)dx\\
		&+2\sum_{1\leq j<l\leq k} \idotsint m(\xi_1,\cdots,\xi_k)f_1^{\sharp}(\xi_1)\cdots f^{\sharp}_k(\xi_k)f_{k+1}^{\sharp}(\xi_{k+1})\\
&\quad\times\tilde{M}_{j,l}(\xi_1,\cdots,\xi_k,\xi_{k+1}) d\xi_1 \cdots d\xi_k d\xi_{k+1}
	\end{split}
\end{equation}

{\bf{Case $s=\frac{1}{2}$}}:
\begin{align*}
&|\xi_{k+1}|^2M(\xi_1,\cdots,\xi_k,\xi_{k+1})\\
&=\int_{\mathbb{R}^{d}} e(x,\xi_1)e(x,\xi_2)\cdots H e(x,\xi_{k+1})dx\\
&=\int_{\mathbb{R}^{d}} V(x)e(x,\xi_1)e(x,\xi_2)\cdots e(x,\xi_{k+1})dx \\
&\qquad - \int_{\mathbb{R}^{d}}  e(x,\xi_{k+1})\Delta[e(x,\xi_1)\cdots e(x,\xi_{k})]dx\\
&=\int_{\mathbb{R}^{d}}V(x) e(x,\xi_1)e(x,\xi_2)\cdots e(x,\xi_{k+1})dx\\
&+\sum_{1\leq j\leq k}\int_{\mathbb{R}^{d}} e(x,\xi_1)e(x,\xi_2)\cdots\nabla e(x,\xi_j) e(x,\xi_{j+1})\cdots \nabla e(x,\xi_{k+1})dx
\end{align*}
Therefore,
\begin{equation}\label{y3.1'}
	\begin{split} &\int_{\mathbb{R}^{d}}H^{\frac{1}{2}} T(f_1,\cdots,f_k)f_{k+1}(x)dx\\
=&\idotsint m(\xi_1,\cdots,\xi_k)f_1^{\sharp}(\xi_1)\cdots f^{\sharp}_k(\xi_k)(|D^{\sharp}|^{-1}f_{k+1})^{\sharp}(\xi_{k+1})\\
&\quad\times|\xi_{k+1}|^{2}M(\xi_1,\cdots,\xi_k,\xi_{k+1}) d\xi_1 \cdots d\xi_k d\xi_{k+1}\\
        =&\int_{\mathbb{R}^{d}}T( f_1,\cdots,f_k)V(x)|D^{\sharp}|^{-1}f_{k+1}(x)dx\\
		&+\sum_{1\leq j\leq k} \idotsint m(\xi_1,\cdots,\xi_k)f_1^{\sharp}(\xi_1)\cdots f^{\sharp}_k(\xi_k)(|D^{\sharp}|^{-1}f_{k+1})^{\sharp}(\xi_{k+1})\\
&\quad\times\tilde{M}_{j,k+1}(\xi_1,\cdots,\xi_k,\xi_{k+1}) d\xi_1 \cdots d\xi_k d\xi_{k+1}
	\end{split}
\end{equation}
where $\tilde{M}_{j,l}(\xi_1,\cdots,\xi_k,\xi_{k+1})=\int_{\mathbb{R}^{d}} e(x,\xi_1)e(x,\xi_2)\cdots\nabla e(x,\xi_j)\cdot \nabla e(x,\xi_l)\cdots e(x,\xi_{k+1})dx$. To the term involving new  nonlinear spectral distribution  $\tilde{M}_{j,l}(\xi_1,\cdots,\xi_k,\xi_{k+1})$, following the arguments  in the proof of  theorem \ref{mthm1.1'} and using interpolation, while  to the rest terms in \eqref{y3.1} and \eqref{y3.1'}, applying theorem \ref{mthm1.1'} and sobolev embedding,  we derive theorem \ref{reyth1.1}.

\subsection{Application 2: Scattering of the generalized mass-critical NLS with good potential for small data in low dimensions}We consider the following  generalized mass-critical nonlinear Schr\"{o}dinger equation with good potential in low dimensions $d=1,2$:
\begin{equation}\label{y3.2.1}
i u_{t}-\Delta u+Vu=a(x)F(u), \quad u(0, x)=u_{0}(x), \quad x\in\mathbb{R}^d.
\end{equation}
when $d=1$, $F(u)=T(\bar{u},\bar{u},u,u,u)(x)$; when $d=2$, $F(u)=T(\bar{u},u,u)(x)$. where recalling that
\begin{equation}
\begin{split}
  T(f_1,f_2,\ldots,f_k)(x):=&\int_{\mathbb{R}^{d}} \ldots\int_{\mathbb{R}^{d}} m\left(\xi_{1}, \xi_{2},\ldots,\xi_k\right) f_1^{\sharp}\left(\xi_{1}\right) f_2^{\sharp}\left(\xi_{2}\right)\ldots f^{\sharp}_k\left(\xi_{k}\right)\\
 &\qquad e\left(x, \xi_{1}\right) e\left(x, \xi_{2}\right)\ldots e\left(x, \xi_{k}\right) \mathrm{d} \xi_{1} \mathrm{d} \xi_{2} \ldots \mathrm{d} \xi_{k},
\end{split}
\end{equation}
and $m$ is a Coifman-Meyer multiplier satisfying \eqref{1.3.1}. Note that the case $m=1$ corresponds (up to a constant factor) to the product of $f_1,\ldots,f_k$. Therefore, in this case, when $V=0$ and $a(x)\equiv1$, or $a(x)\equiv-1$, the equation \eqref{y3.2.1} becomes a classical mass-critical nonlinear Schr\"{o}dinger equation in $d=1,2$. For good potential $V$: $V$ satisfies \(\mathrm{H} 1, \mathrm{H} 2,\) and \(\mathrm{H} 3 \), and assume that the Riesz transform \(\Re=\nabla(-\Delta+V)^{-1 / 2}\) is bounded on $L^{p}, 1<p<\infty$, we have the scattering of the generalized mass-critical NLS with good potential for small data in low dimensions $d=1,2$.

\begin{theorem}[local wellposedness and  small data scattering]\label{th-y3.3} For $d=1,2$, $a(x)\in L^{\infty}$, the equation \eqref{y3.2.1} has the following properties:
\begin{enumerate}
  \item (Local wellposedness) For any $u_{0} \in L_{x}^{2}\left(\mathbf{R}^{d}\right)$, there exists $T\left(u_{0}\right)>0$ such that \eqref{y3.2.1} is locally well posed on $[-T, T].$ The term $T\left(u_{0}\right)$ depends on the profile of the initial data as well as its size. Moreover, \eqref{y3.2.1} is well posed on an open interval $I \subset \mathbf{R}$, $0 \in I$;
  \item (Small data scattering) there exists $\varepsilon_{0}(d)>0$, such that if
\begin{equation}\label{yeq3.31}
\left\|u_{0}\right\|_{L^{2}\left(\mathbf{R}^{d}\right)} \leq \varepsilon_{0}(d),
\end{equation}
then \eqref{y3.2.1} is globally well posed and scattering, i.e. there exist $u^{\pm} \in L_{x}^{2}(\mathbb{R}^d )$ such that
\begin{equation} \label{eq1.3''}
\|u(t)- e^{it\Delta} u^{\pm}\|_{L_{x}^{2}} \to 0, \ \text{ as } t\to \pm \infty.
\end{equation}
\end{enumerate}
\end{theorem}

\begin{proof}
With the Strichartz estimate for the Schr\"{o}dinger operator $H=-\Delta + V$ in section 2, we can obtain local wellposedness and small data scattering for \eqref{y3.2.1} by the contraction mapping principle and bootstrap argument. Those steps are standard. We recommend to refer to Section 1.3 in \cite{D1} for more details. Due to the different nonlinear terms, we give the nonlinear estimates that may be used below:
\begin{align*}
\left\|\int_{0}^{t} e^{itH} a(x)F(u(\tau)) d \tau\right\|_{L_{t, x} ^{\frac{2(d+2)}{d}}\left(\mathbf{R} \times \mathbf{R}^{d}\right)} & \lesssim_d\|F(u)\|_{L_{t, x}^{\frac{2(d+2)}{d+4}}\left(\mathbf{R} \times \mathbf{R}^{d}\right)} \\
& \lesssim\|u\|^{1+\frac{4}{d}}_{L_{t, x} ^{\frac{2(d+2)}{d}}\left(\mathbf{R} \times \mathbf{R}^{d}\right)},
\end{align*}
and for example, when $d=2$, $T(\bar{u},u,u)(x)-T(\bar{v},v,v)(x)= T(\bar{u}-\bar{v},u,u)(x)+T(\bar{v},u-v,u)(x)+ T(\bar{v},v,u-v)(x)$, therefore
\begin{align*}
&\|a(x)(F(u)-F(v))\|_{L_{t, x}^{\frac{2(d+2)}{d+4}}\left(\mathbf{R} \times \mathbf{R}^{d}\right)}\\
\lesssim & \left(\|u\|^{\frac{4}{d}}_{L_{t, x} ^{\frac{2(d+2)}{d}}\left(\mathbf{R} \times \mathbf{R}^{d}\right)}+\|v\|^{\frac{4}{d}}_{L_{t, x} ^{\frac{2(d+2)}{d}}\left(\mathbf{R} \times \mathbf{R}^{d}\right)}\right) \| u-v \|_{L_{t, x} ^{\frac{2(d+2)}{d}}\left(\mathbf{R} \times \mathbf{R}^{d}\right)}.
\end{align*} Following the standard argument, we first get scattering for \eqref{y3.2.1} with respect to the Schr\"{o}dinger operator $H=-\Delta + V$ as follows,
\begin{equation} \label{eq1.3'''}
\|u(t)- e^{itH} u^{\pm}\|_{L_{x}^{2}} \to 0, \ \text{ as } t\to \pm \infty.
\end{equation}
In addition, since wave operator $\Omega$ exists and is complete, we finally get scattering for \eqref{y3.2.1} in the sense of \eqref{eq1.3''}.
\end{proof}

\section*{Acknowledgments}
This work does not have any conflicts of interest. The author is supported by the China Postdoctoral Science Fundation (No.2019M660556) and Doctoral Fundation of Chongqing Normal University (No.21XLB025).

\begin{center}

\end{center}

\end{document}